\documentclass[12pt, a4paper, reqno]{amsart}

\usepackage{amsmath, amssymb, amsthm}
\usepackage[colorlinks=true, urlcolor=blue, linkcolor=blue, citecolor=blue, pdfstartview=FitH]{hyperref}
\usepackage{esint}
\usepackage{tikz-cd}

\newtheorem{theorem}{Theorem}
\newtheorem{lemma}{Lemma}
\newtheorem{proposition}{Proposition}

\newtheorem{corollary}{Corollary}

\newtheorem{remark}{Remark}

\numberwithin{equation}{section}

\title[A sharp estimate for the genus of surfaces]{A sharp estimate for the genus of embedded surfaces in the $3$-sphere}
\begin{document}
\author{Kwok-Kun Kwong}
\address{
Mathematical Sciences Institute\\
Australian National University\\
ACT 0200, Australia}
\email{kwok-kun.kwong@anu.edu.au}
\maketitle

\begin{abstract}
By refining the volume estimate of Heintze and Karcher \cite{HK}, we obtain a sharp pinching estimate for the genus of a surface in $\mathbb S^{3}$, which involves an integral of the norm of its traceless second fundamental form. More specifically, we show that if $g$ is the genus of a closed orientable surface $\Sigma$ in a $3$-dimensional orientable Riemannian manifold $M$ whose sectional curvature is bounded below by $1$, then
$4 \pi^{2} g(\Sigma) \le 2\left(2 \pi^{2}-|M|\right)+\int_{\Sigma} f(|\stackrel \circ A|)$, where $ \stackrel \circ A $ is the traceless second fundamental form and $f$ is an explicit function. As a result, the space of closed orientable embedded minimal surfaces $\Sigma$ with uniformly bounded $\|A\|_{L^3(\Sigma)}$ is compact in the $C^k$ topology for any $k\ge2$.
\end{abstract}

\section{Introduction}

Suppose $\Sigma$ is an umbilical surface in the sphere $\mathbb S^{3}$, i.e. the difference of the principal curvatures $k_2-k_1\equiv 0$, then $\Sigma$ is a geodesic sphere. Our result shows that in a $3$-manifold with sectional curvature bounded below by $1$, if the two principal curvatures $k_1$, $k_2$ are close enough in an integral sense, then it is homeomorphic to a sphere. More generally, we can show that if $M$ is a closed Riemannian $3$-manifold whose sectional curvature is bounded below by $1$, then a certain integral involving the gap between the two principal curvatures of a closed surface controls the genus of $\Sigma$. See Theorem \ref{thm2}.

For example, as a corollary of our result, we can give a sharp bound of the genus of a surface in $\mathbb S^{3}$:
\begin{theorem} [Corollary \ref{cor1}]
For any closed embedded surface $\Sigma$ of genus $g$ in $\mathbb S^{3}$,
\begin{equation*}
4 \pi^{2} g(\Sigma) \le \int_{\Sigma}f\left(|\stackrel \circ A|\right).
\end{equation*}
Here $\stackrel \circ A$ is the traceless second fundamental form of $\Sigma$ and $f(t)=\sqrt{2} t+\left(t^{2}-2\right) \tan^{-1}\left(\frac{t}{\sqrt{2}}\right)$.

The equality holds if and only if $\Sigma$ is a geodesic sphere or a Clifford torus.
\end{theorem}
As a result of this pinching estimate, by using a compactness result of Choi-Schoen \cite{choi1985space}, we can prove the following compactness result:
\begin{corollary}[Corollary \ref{cpt}]
Let $M$ be a closed orientable three-dimensional Riemannian manifold whose sectional curvature of $M$ is bounded from below by $1 $ and $C\ge 0$.
Then the space of closed orientable embedded minimal surfaces $\Sigma$ with $ \|A\|_{L^3(\Sigma)} \le C$ is compact in the $C^k$ topology for any $k\ge2$.
\end{corollary}
As another corollary, we can prove the following $L^3$ gap theorem for minimal surfaces in $\mathbb S^3$.
\begin{corollary}[Corollary \ref{gap}]
A closed embedded minimal surface $\Sigma$ in $\mathbb S^{3}$ with $\int_\Sigma |A|^3<3\sqrt{2}\pi^2$ is an equator.
\end{corollary}

There is a number of existing results of pinching estimates and rigidity theorems for surfaces in $\mathbb R^3$ or $\mathbb R^n$ in the literature. For example, Pogorelov \cite[p. 493]{PAV} proved that if $\Sigma$ is a closed convex surface and the ratio of its principal curvatures are uniformly close to $1$, then $\Sigma$ is close to a round sphere. In \cite{vodopyanov1970estimates}, a similar result is proved by replacing the pointwise condition by some integral condition. In \cite{de2005optimal, de2006ac}, the authors proved that if the $L^2 $ norm of the traceless second fundamental form is small enough, then the surface is $W^{2, 2}$ or $C^0$ close to a sphere. See also \cite[Ch. 6]{reshetnyak2013stability} for a survey on this matter.

There are works on the rigidity theorems in other space forms as well. Most noticably, there are pioneering rigidity theorems for minimal submanifolds in a sphere, due to Simons \cite{simons1968minimal}, Lawson \cite{Lawson1970complete} and Chern-do Carmo-Kobayashi \cite{chern1970minimal}, in which rigidity for oriented closed minimal submanifolds in $\mathbb S^{n+p}$ with $ |A|^2\le \frac{n}{2-\frac{1}{p}}$ is obtained. For pinching estimates, in \cite[Theorem 1.3]{cheng2012rigidity} (cf. also \cite{Perez}), for a closed hypersurface in the standard space forms, a pinching result of the variance of the mean curvature $H$ is proved, which involves a bound on the $L^2$ norm of the traceless second fundamental form. On the other hand, this estimate is modelled on the umbilical hypersurfaces (i.e. geodesic spheres) and no topological or geometric information is known if the $L^2$ norm of the traceless second fundamental form is large, whereas our result can be used to give some topological information about the surface given a bound on a certain integral of the norm of the traceless second fundamental form.

The organization of the remaining sections in this paper is as follows. In Section \ref{sec 2}, we present a proof of an estimate for the genus of a surface embedded in a three-manifold with curvature bounded below by $1$.
Section \ref{sec 3} establishes the connection between this estimate and the norm of the traceless second fundamental form, establishing the main result. Furthermore, we explore some applications, including a compactness result regarding the space of minimal surfaces in $\mathbb S^3$.
In the Appendix, Section \ref{sec 4 appendix}, we conduct a comparative analysis between our estimate and the classical estimate involving the area and the $L^2$ norm of the traceless second fundamental form. By providing numerical examples, we demonstrate that our estimate cannot be derived from the classical estimate and, in fact, performs consistently better in all the computed cases.

\section{An estimate of the genus}\label{sec 2}

To prove the main result, we will need the following proposition, which results from a refinement of a volume estimate by Heintze and Karcher.

\begin{proposition}\label{prop 1}
Let $\Sigma$ be a closed orientable embedded surface in a closed orientable three-dimensional Riemannian manifold $M$. Assume the sectional curvature of $M$ is bounded from below by $1$.
Let $k_1\le k_2$ be the principal curvatures of $\Sigma$ (w.r.t. some unit normal), then
\begin{align*}
4 \pi^{2} g(\Sigma)\le 2\left(2 \pi^{2}-|M|\right) +\int_{\Sigma}\left(k_{2}-k_{1}-(1+k_1k_2)\left(\tan^{-1} k_{2}-\tan^{-1} k_{1}\right)\right).
\end{align*}

Here $g(\Sigma)$ is the genus of $\Sigma$.
\end{proposition}

The following lemma is sketched in \cite{choi1983first} and is well-known to experts. For reader's convenience, we provide slightly more details here.
\begin{lemma}\label{lem sep}
With the assumption in Proposition \ref{prop 1}, $\Sigma$ separates $M$ into two components.
\end{lemma}
\begin{proof}
This is equivalent to showing that $H_{0}(M \backslash \Sigma) \cong \mathbb Z \oplus \mathbb Z$.

As the Ricci curvature of $M$ is positive, its first Betti number $b_1(M)$ is zero. By deRham's theorem, $H^1(M, \mathbb R)=0$, and hence by universal coefficient theorem, $\mathrm{rank}(H_1(M))=0$ as a finitely generated Abelian group.
By Poincare duality, the cohomology group $H^{2}(M ; \mathbb Z)$ has rank zero. Since $M$ is compact, the cohomology groups are all finitely generated Abelian groups and consequently we have that every element of $H^{2}(M ; \mathbb Z)$ is torsion.
Consider the exact sequence (\cite[p. 200]{hatcher2005algebraic})
$$
\begin{tikzcd}
& & {\mathbb Z} \arrow[d, equal] & &
\\
{}&H^{2}(M ; \mathbb Z)\arrow[r, "\phi"] &H^{2}(\Sigma ; \mathbb Z)\arrow[r, "f"] & H^{3}(M, \Sigma ; \mathbb Z) \\
\arrow[r, "g"] & H^{3}(M ; \mathbb Z) \arrow[r]&H^{3}(\Sigma ; \mathbb Z)&{}\\
& {\mathbb Z} \arrow[u, equal] &0\arrow[u, equal]& &
\end{tikzcd}
$$

Since $\mathbb Z$ has no non-zero torsion element, we have $\mathrm{im}(\phi)=0$. Using the first isomorphism theorem and exactness we deduce that
$\mathrm{ker}(g)=\mathrm{im}(f) \cong \mathbb Z / \mathrm{ker}(f) =\mathbb Z / \mathrm{im}(\phi)\cong \mathbb Z$.
Applying the first isomorphism theorem to $g$, we get the isomorphism
$H^{3}(M, \Sigma; \mathbb Z) / \mathbb Z \cong \mathrm{im}(g)\cong \mathbb Z$,
from which it follows that $H^{3}(M, \Sigma ; \mathbb Z) \cong \mathbb Z \oplus \mathbb Z$. By Poincare-Alexander-Lefschetz duality (\cite[Theorem 8.3]{bredon2013topology}), we then get the isomorphism
$$ H_{0}(M \backslash \Sigma) \cong H^{3}(M, \Sigma ; \mathbb Z) \cong \mathbb Z \oplus \mathbb Z. $$
\end{proof}

\begin{proof}[Proof of Proposition \ref{prop 1}]
By Lemma \ref{lem sep}, $\Sigma$ separates $M$ into two components, $\Omega_1$ and $\Omega_2$.

It is clear that our inequality does not depend on the choice of the normal, so we can without loss of generality assume $\{k_i\}_{i=1}^{2}$ (with $k_1\le k_2$) are the principal curvatures of $\Sigma$ w.r.t. the unit outward normal of $\Omega_1$. By \cite[Corollary 3.3.1]{HK} (cf. also \cite[(10''), (11'')]{montiel1991compact}), the volume of $\Omega_1$ is bounded from above:
\begin{equation*}\label{ineq1'}
\begin{split}
|\Omega_1|
\le& \int_{\Sigma}\int_{0}^{c(p)}(\cos t - k_1\sin t)(\cos t - k_2\sin t) dt dS(p),
\end{split}
\end{equation*}
where $c(p)$ is the cut distance of $\Sigma$ w.r.t. $\Omega_1$. In \cite{HK}, the integrand is estimated by using the AM-GM inequality. In order to obtain a more precise estimate, we keep the integrand unchanged.
The cut distance at each point is bounded by above by the focal length, which is not longer than $\cot^{-1}(k_2)$ by Jacobi field comparison, as we have a lower bound on the sectional curvature. So direct integration gives
\begin{equation*}\label{ineq1}
\begin{split}
|\Omega_1|
\le& \int_{\Sigma}\int_{0}^{\cot^{-1}(k_2)}(\cos t - k_1\sin t)(\cos t - k_2\sin t) dt dS\\
=& \frac{1}{2}\int_{\Sigma}(-k_1+(1+k_1k_2)\cot^{-1}(k_2)).
\end{split}
\end{equation*}
Similarly, with respect to the outward unit normal of $\Omega_2$, the principal curvatures are $-k_2\le -k_1$, so we have
\begin{align*}\label{ineq2}
|\Omega_2| \le\frac{1}{2} \int_{\Sigma}(k_2+ (1+k_1k_2)\cot^{-1}(-k_1)).
\end{align*}
By adding the two inequalities and using the identity $\tan^{-1} x=\frac{\pi}{2}-\cot^{-1} x$,
\begin{equation}\label{sum}
\begin{split}
2|M| \le&\int_{\Sigma}(k_2-k_1+(1+k_1k_2)(\cot^{-1}(k_2)+\cot^{-1}(-k_1)))\\
=&\int_{\Sigma}\left(k_2-k_1+(1+k_1k_2)\left(\pi-\left(\tan^{-1}(k_2)-\tan^{-1}(k_1)\right) \right)\right)\\
\le&\int_{\Sigma}\left(k_2-k_1+ \pi K-(1+k_1k_2)\left(\tan^{-1}(k_2)-\tan^{-1}(k_1)\right) \right).
\end{split}
\end{equation}
By the Gauss-Bonnet theorem, the result follows. Note also that $|M|\le |\mathbb S^3|=2\pi^2$ by the volume comparison theorem.
\end{proof}

\section{Applications}\label{sec 3}

Our aim is to draw some conclusions from Proposition \ref{prop 1}.

\begin{lemma}\label{lem3}
For any $k_2\ge k_1$,
\begin{equation*}
-\left(1+k_{1} k_{2}\right)\left(\tan^{-1} k_{2}-\tan^{-1} k_{1}\right) \le 2\left(-1+\left(\frac{k_{2}-k_{1}}{2}\right)^{2}\right) \tan^{-1}\left(\frac{k_{2}-k_{1}}{2}\right).
\end{equation*}
The equality holds if and only if $k_1=\pm k_2$.
\end{lemma}

\begin{proof}
Let $k_{2}-k_{1}=2 t $ and $ k_{2}+k_{1} =2s $. It is equivalent to show that
$f (t, s)=2\left(t^{2}-1\right) \tan^{-1} t+\left(1+ s^{2}-t^{2} \right)\left[\tan^{-1}(s +t)-\tan^{-1}(s -t)\right] \ge 0$ for $t \ge 0$. By symmetry we can also assume $s \ge 0$.

This is an equality when $s =0$, so it suffices to show that $\frac{\partial f}{\partial s }\ge 0$.
We compute
\begin{align*}
\frac{\partial f}{\partial s }
=&2s \left(\tan^{-1}(s +t)- \tan^{-1}(s -t)\right)+\left(1+s^{2}-t^{2}\right)\left(\frac{1}{1+(t+s)^{2}}-\frac{1}{1+(t-s)^{2}}\right).
\end{align*}
As $\frac{\partial f}{\partial s }(0, s)=0$, it in turn suffices to show that $\frac{\partial }{\partial t}\frac{\partial f}{\partial s }\ge 0$. By direct computation,
\begin{align*}
\frac{\partial}{\partial t} \frac{\partial f}{\partial s }
=\frac{32 t^{2} s \left(1+t^{2}+s^{2}\right)}{\left(1+(t-s)^{2}\right)^{2} \left(1+(t+s)^{2}\right)^{2}}\ge0.
\end{align*}
Therefore the inequality holds.
From the above, the equality holds if and only if $s =0$ or $t=0$, which is equivalent to $k_1=\pm k_2$.
\end{proof}
\begin{theorem}\label{thm2}
With the same assumption and notation in Proposition \ref{prop 1}, we have
\begin{equation}\label{est}
4 \pi^{2} g(\Sigma)
\le2\left(2 \pi^{2}-|M|\right)+\int_{\Sigma}\left[\sqrt{2}|\stackrel{\circ}{A}|+\left(|\stackrel {\circ}{A}|^{2}-2\right) \tan^{-1}\left(\frac{|\stackrel {\circ}{A}|}{\sqrt{2}}\right)\right].
\end{equation}
If the equality holds, then $\Sigma$ is either umbilical or is a minimal surface. Moreover, the cut distance (from both sides of $\Sigma$) at each point of $\Sigma$ must equal to the focal length.
\end{theorem}

\begin{proof}
Let $k_1\le k_2$ be the principal curvatures, then we have $ |\stackrel {\circ}{A}|=\frac{k_2-k_1}{\sqrt{2}}$.
By Proposition \ref{prop 1}
\begin{align*}
4 \pi^{2} g(\Sigma)\le\int_\Sigma \left(k_2-k_1-(1+k_1k_2)(\tan^{-1}k_2-\tan^{-1}k_1)\right).
\end{align*}
So by Lemma \ref{lem3},
\begin{align*}
4 \pi^{2} g(\Sigma)
\le& \int_\Sigma (k_2-k_1)+2\int_{\Sigma}\left(-1+\left(\frac{k_2-k_1}{2} \right)^2\right) \tan^{-1} \left(\frac{k_2-k_1}{2}\right)\\
=&\int_{\Sigma} \left[\sqrt{2}|\stackrel {\circ}{A}|+\left(|\stackrel {\circ}{A}|^{2}-2\right) \tan^{-1}\left(\frac{ | \stackrel{\circ}{A}|}{\sqrt{2}}\right)\right].
\end{align*}

We remark that $\sqrt{2}|\stackrel {\circ}{A}|+\left(|\stackrel {\circ}{A}|^{2}-2\right) \tan^{-1}\left(\frac{ | \stackrel{\circ}{A}|}{\sqrt{2}}\right)$ is strictly increasing in $|\stackrel {\circ}{A}|$. Indeed, for $f(t)=
\sqrt{2}t+\left(t^{2}-2\right) \tan^{-1}\left(\frac{t}{\sqrt{2}}\right) $, $f'(t)=\frac{2 \sqrt{2} t^{2}}{2+t^{2}}+2 t \tan^{-1}\left(\frac{t}{\sqrt{2}}\right)>0$ for $t>0$. In particular, the integrand on the RHS is always non-negative.

We now assume the equality holds. By \cite[3.4.8]{HK}, all planes of $M$ containing a tangent vector to a geodesic segment emanating orthogonally from $\Sigma$ have sectional curvature $1$, up to its focal length. \eqref{sum} also shows that $R^M_{1221}=1$ and so the sectional curvature along $\Sigma$ must be constantly equal to $1$.
\end{proof}
In view of Theorem \ref{thm2}, in the following, we define the function
\begin{equation}\label{f}
f(t)=\sqrt{2} t+\left(t^{2}-2\right) \tan^{-1}\left(\frac{t}{\sqrt{2}}\right).
\end{equation}

\begin{remark}\label{rmk1}
The inequality in Theorem \ref{thm2} can also be written in a form such that the integrand is expressed as a power series of $|\stackrel \circ A|$:
\begin{align*}
4 \pi^{2} g(\Sigma)\le& 2\left(2 \pi^{2}-|M|\right)+\int_{\Sigma}f(|\stackrel{\circ}{A}|)\\
=&2\left(2 \pi^{2}-|M|\right)+\int_{\Sigma}\left[\sum_{l=1}^{\infty}(-1)^{l+1} \frac{8 l}{4 l^{2}-1}\left(\frac{|\stackrel {\circ}{A}|}{\sqrt{2}}\right)^{2 l+1}\right].
\end{align*}
It is not hard to check that $f$ is a strictly increasing convex function with $f(0)=0$.

Observe that the Taylor series above is alternating and starts from the third order term. Indeed we have the following lemma.
\begin{lemma}\label{lem4}
For $t\ge 0$, $f(t) \le \frac{2 \sqrt{2} t^{3}}{3}$ and the inequality is strict unless $t=0$.
\end{lemma}
\begin{proof}
For $t>0$, we have
\begin{align*}
\frac{d}{dt}\left(\frac{2 \sqrt{2} t^{3}}{3}-f(t)\right)
=&\frac{2 t}{2+t^{2}}\left(\sqrt{2} t\left(1+t^{2}\right)-\left(2+t^{2}\right) \tan^{-1}\left(\frac{t}{\sqrt{2}} \right)\right)\\
\ge&\frac{2 t}{2+t^{2}}\left(\sqrt{2} t\left(1+t^{2}\right)-\left(2+t^{2}\right)\left(\frac{t}{\sqrt{2}} \right)\right)>0.
\end{align*}
\end{proof}

By Lemma \ref{lem4}, instead of the bound in Theorem \ref{thm2}, we can bound the genus by a weaker but simpler expression:
\begin{corollary}\label{cor1'}
With the same assumption and notation in Proposition \ref{prop 1}, we have
$$
2 \pi^{2} g(\Sigma)
\le 2 \pi^{2}-|M| +\frac{\sqrt{2}}{3}\int_{\Sigma}|\stackrel{\circ}{A}|^3.
$$

If the equality holds, then $\Sigma$ is umbilical.
\end{corollary}
\end{remark}
\begin{corollary}\label{cor1}
For any closed surface $\Sigma$ in $\mathbb S^{3}$, we have
\begin{enumerate}
\item \label{point 1}
$$
4 \pi^{2} g(\Sigma) \le \int_{\Sigma}\left[\sqrt{2}|\stackrel{\circ}{A}|+\left(|\stackrel{\circ}{A}|^{2}-2\right) \tan^{-1}\left(\frac{|\stackrel{\circ}{A}|}{\sqrt{2}}\right)\right].
$$
The equality holds if and only if $\Sigma$ is a geodesic sphere or is a Clifford torus.
\item
$$
2 \pi^{2} g(\Sigma) \le \frac{ \sqrt{2}}{3} \int_{\Sigma}|\stackrel{\circ}{A}|^{3}.
$$
The equality holds if and only if $\Sigma$ is a geodesic sphere.
\end{enumerate}
\end{corollary}

\begin{proof}
By Theorem \ref{thm2} and Corollary \ref{cor1'}, we only have to consider the equality case. By Lemma \ref{lem4}, it suffices to prove \eqref{point 1}.

By Lemma \ref{lem3}, at each point in $\Sigma$, we have either $k_1=k_2$ or $k_1+k_2=0$. Suppose there is an umbilical point $p$ such that the principal curvatures $k_i\ne 0$. Then by continuity, $k_1=k_2$ in a neighborhood of $p$. By the usual proof of the Schur's lemma by using the Codazzi equation, we conclude that on this neighborhood, $k_i=\textrm{constant}\ne 0$. By a standard connectedness argument we then conclude that $\Sigma$ has constant principal curvatures and is a geodesic sphere. The remaining possibility is of course $k_1+k_2\equiv 0$, i.e. $\Sigma$ is a minimal surface.

Suppose now $\Sigma$ is a minimal surface. From the proof of Proposition \ref{prop 1}, the cut distance (from both sides of $\Sigma$) at each point must also equal to the focal length $\cot^{-1}k_2$. Assume there exists a point such that $k_1=k_2=0$, then the focal point of $\Sigma$ inside $\Omega_1$ is also a cut point, and so its distance from $\Sigma$ is $\cot^{-1}(0)=\frac{\pi}{2}$. By \cite[Theorem A(2)]{kasue1983Ricci}, $\Omega_1$ is the hemisphere and so $\Sigma$ is the equator. Therefore, if $\Sigma$ is not the equator, then $k_2>0$ on $\Sigma$. By \cite[Proposition 1.5]{Lawson1970complete}, $\Sigma$ must be a torus. By Brendle's proof of the Lawson's conjecture \cite{brendle2013embedded}, $\Sigma$ is a Clifford torus.
\end{proof}

For any surface, we have $|\stackrel {\circ}{A}|^{2}=|A|^{2}-2 H^{2}$, where $H=\frac{1}{2}(k_1+k_2)$ is the normalized mean curvature.
In particular, $|\stackrel {\circ}{A}|^{2}=|A|^{2}$ on a minimal surface.

\begin{corollary}
For any closed embedded minimal surface $\Sigma$ in $\mathbb S^{3}$, we have
\begin{equation*}\label{min surf}
4 \pi^{2} g(\Sigma)\le\int_{\Sigma}\left[\sqrt{2}| A |+\left(| A |^{2}-2\right) \tan^{-1}\left(\frac{| A |}{\sqrt{2}}\right)\right].
\end{equation*}
The equality holds if and only if $\Sigma$ is an equator or is a Clifford torus.
\end{corollary}
This result combined with Lemma \ref{lem4} gives a gap theorem:
\begin{corollary}\label{gap}
A closed embedded minimal surface $\Sigma$ in $\mathbb S^{3}$ with $\int_\Sigma |A|^3<3\sqrt{2}\pi^2$ is an equator.
\end{corollary}
\begin{corollary} \label{cor2}
If $\Sigma$ is a closed embedded surface in $\mathbb S^3$ such that $\displaystyle |\stackrel {\circ}{A}|<\sqrt{2} \beta $, where $\beta$ is the unique solution to
\begin{align*}
\beta+(\beta^2-1)\tan^{-1}\beta =
\frac{2 g_0\pi^{2}}{|\Sigma|}
\end{align*}
for some $ g_0\in\mathbb N$, then the genus of $\Sigma$ is less than $g_0$.
In particular, if
$\displaystyle \beta+(\beta^2-1)\tan^{-1}\beta = \frac{2\pi^2}{|\Sigma|}$ and $\displaystyle |\stackrel{\circ}{A}|<\sqrt{2} \beta $,
then $\Sigma$ is homeomorphic to a sphere.
\end{corollary}
\begin{remark}
The condition in Corollary \ref{cor2} cannot be weakened to $k_1-k_1\le 2\beta $. Indeed, for the Clifford torus $\Sigma$, we have $k_1=-1, k_2=1$ and $|\Sigma|=2\pi^2$. So in this case we have $2\beta=2$, $\chi(\Sigma)=0$ and $k_2-k_1=2$.
\end{remark}

Yang and Yau \cite{yang1980eigenvalues} proved the following eigenvalue estimate:
Let $\Sigma$ be a compact Riemannian surface of genus $g$, then for any metric on $\Sigma$,
\begin{equation}\label{yy}
\lambda_{1} (\Sigma){\mathrm{Area}\left(\Sigma\right)}\le 8 \pi(g+1)
\end{equation}
where $\lambda_1$ is the first Laplacian eigenvalue of $\Sigma$. It was observed by El Soufi and Ilias \cite{el1983volume} that the same proof gives the improved bound
\begin{equation}\label{improve}
\lambda_{1} (\Sigma){\mathrm{Area}\left(\Sigma\right)}\le 8 \pi \left\lfloor\frac{g+3}{2}\right\rfloor,
\end{equation}
where $\lfloor\cdot\rfloor$ is the floor function. See also \cite{karpukhin2021improved} for a further improvement.

Using \eqref{yy} and Theorem \ref{thm2}, we have
\begin{corollary}
With the same assumption and notation in Proposition \ref{prop 1}, we have
$$
\lambda_{1}(\Sigma) \mathrm{Area}(\Sigma) \le 16 \pi-\frac{4|M|}{\pi}+\frac{2}{\pi} \int_{\Sigma} f(|\stackrel{\circ}{A}|)
$$
where $f(t)=\sqrt{2} t+\left(t^{2}-2\right) \tan^{-1}\left(\frac{t}{\sqrt{2}}\right)$.
\end{corollary}
When $M$ is the $3$-sphere, we have the following
\begin{corollary}
For a closed embedded surface in $\mathbb{S}^{3}$, we have
$$
\lambda_{1}(\Sigma) \mathrm{Area}(\Sigma) \le 8 \pi+\frac{2}{\pi} \int_{\Sigma} f(|\stackrel{\circ}{A}|).
$$
The equality holds if and only if $\Sigma$ is a geodesic sphere or is a Clifford torus.
\end{corollary}

The following theorem is due to Choi and Schoen.
\begin{theorem} \cite[Theorem 2]{choi1985space}
Assume $M$ is a closed 3-dimensional manifold with positive Ricci curvature. There exists a constant $C$ depending only on $M$ and an integer $g$ such that if $\Sigma$ is a closed embedded minimal surface of genus $g$ in $M$, then
$$
\max_{\Sigma}|A| \le C.
$$
\end{theorem}
It is interesting to compare it with the following result:
\begin{corollary}
Assume $M$ is a closed orientable 3-dimensional manifold whose sectional curvature is bounded below by $1$ and $\Sigma$ is a closed embedded orientable minimal surface of genus $g\ge 1$ in $M$, then we have
\begin{align*}
\max_\Sigma|A|\ge f^{-1}\left(\frac{2 \pi^{2}(g-1)+|M|}{4 \pi\left\lfloor\frac{g+3}{2}\right\rfloor}\right)
\end{align*}
where $f$ is given by \eqref{f}. In paricular, $\displaystyle \max_{\Sigma}|A| \ge C>0$, where $C$ depends on the genus $g$ and the volume of $M$ only.
\end{corollary}
\begin{proof}
For a minimal surface, $|A|=|\stackrel \circ A|$.
As remarked above, $f$ is strictly increasing, and so by Theorem \ref{thm2},
\begin{equation*}
\begin{split}
4 \pi^{2} g \le 2(2\pi^2-|M|)+\int_{\Sigma} f(|A|)\le 2\left(2 \pi^{2}-|M|\right)+ \mathrm{Area}(\Sigma) f(\max_\Sigma |A|).
\end{split}
\end{equation*}
By the proof of the area estimate in \cite[Proposition 4]{choi1983first} and the improvement \eqref{improve}, we have $\mathrm{Area}(\Sigma)\le 8 \pi\left\lfloor\frac{g+3}{2}\right\rfloor$. Therefore
\begin{align*}
\max_\Sigma|A|\ge f^{-1}\left(\frac{4\pi^2 g-2(2\pi^2-|M|)}{8\pi\left\lfloor\frac{g+3}{2}\right\rfloor}\right)
=f^{-1}\left(\frac{2\pi^{2} (g-1)+|M|}{4 \pi\left\lfloor\frac{g+3}{2}\right\rfloor}\right).
\end{align*}
In particular, if $M=\mathbb S^3$, then
$$
\max_{\Sigma}|A| \ge f^{-1}\left(\frac{\pi}{2}\cdot\frac{ g}{ \left\lfloor\frac{g+3}{2}\right\rfloor}\right).
$$
\end{proof}
\begin{corollary}\label{cpt}
Let $M$ be a closed orientable three-dimensional Riemannian manifold whose sectional curvature of $M$ is bounded from below by $1 $ and $C\ge 0$.
Then
\begin{enumerate}
\item
The space of closed orientable embedded minimal surfaces $\Sigma$ with $\int_\Sigma f(|A|)\le C$ is compact in the $C^k$ topology for any $k\ge2$.
\item
The space of closed orientable embedded minimal surfaces $\Sigma$ with $\int_\Sigma |A|^3 \le C$ is compact in the $C^k$ topology for any $k\ge2$.
\end{enumerate}
\end{corollary}
\begin{proof}
By Theorem \ref{thm2} and Corollary \ref{cor1'}, both spaces are contained in the space of closed orientable embedded minimal surfaces with an upper bound of the genus and are closed in the $C^k$ topology, $k\ge2$. Therefore they are both compact by \cite[Theorem 1]{choi1985space}.
\end{proof}

\section{Appendix: Comparison with the classical estimate}\label{sec 4 appendix}
In this appendix, we present a comparison between our estimate and a more well-known estimate involving the $L^2$ norm of $ \stackrel \circ A $.

It is well-known that the genus of $\Sigma$ can also be bounded in terms of $\int_\Sigma |\stackrel \circ A|^2 $ when $M$ has sectional curvature at least $1$.
To see this, note that the Gaussian curvature $K_{\Sigma}$ of $\Sigma$ satisfies $K_{\Sigma} \ge 1+k_1 k_2$ by the Gauss equation and the assumption that the sectional curvature of $M$ is at least $1 $. As $-k_1 k_2 \le \frac{1}{2}|\stackrel \circ A|^2$, by integrating this inequality and using the Gauss-Bonnet formula, we have
\begin{equation}\label{est'}
\begin{split}
4 \pi g & =4 \pi-\int_{\Sigma} K_{\Sigma} \\
& \le 4 \pi-|\Sigma|+\frac{1}{2} \int_{\Sigma}|\stackrel \circ A|^2.
\end{split}
\end{equation}
It appears that this estimate is simpler compared to \eqref{est}, and the $L^2$ norm of $|\stackrel \circ A|$ in the estimate seems to be preferable to the $L^3$ norm of $|\stackrel \circ A|$ as seen in Corollary \ref{cor1'}.

However, it is important to note that this inequality does not imply our estimate \eqref{est}. In fact, we have verified this through numerous numerical examples, which indicate that \eqref{est} is better than \eqref{est'}. This observation demonstrates that (i) our estimate is novel and (ii) it is not possible to derive \eqref{est} from \eqref{est'}. Note also that we do not impose any assumption on the area in Theorem \ref{thm2}.

To compare the two estimates, we have computed examples using ``ellipsoids'' in $\mathbb S^3$. Upon examining the proof of \eqref{est'}, it is evident that it cannot be sharp for ``small'' surfaces that are nearly umbilical, which have genus $0$. (However, it can still be sharp after taking the floor on the right-hand side.)

To aid our computation, we observe that the surface integral $\int_{\Sigma}|\stackrel \circ A|^2 dS$ is conformally invariant. Specifically, the spherical metric can be expressed as $g_{\mathbb S^3}=\frac{4}{(1+|x|^2)^2}\overline g$ under the stereographical projection, where $\overline g$ represents the flat metric on $\mathbb R^3$. Therefore, we can calculate
\begin{equation*}
{A}_i^j=\frac{1+r^2}{2} \overline A_i^j-r \overline g\left(\partial_r, \nu\right) \delta_i^j,
\end{equation*}
where $r=|x|$ and $\overline A$ denotes the second fundamental form with respect to $\overline g$. From this, we see that $|\stackrel{\circ}{{A}}|= \frac{1+r^2}{2} |\stackrel{\circ}{\overline A}|$ and $\int_{\Sigma}|\stackrel \circ A|^2 dS_{g_{\mathbb S^3}}=
\int_{\Sigma}|\stackrel \circ {\overline A}|^2 dS_{\overline g} $. Additionally, note that $|\Sigma|_{g_{\mathbb S^3}}=\frac{4}{(1+r^2)^2}|\Sigma|_{\overline g}$, and so we can freely rescale $\Sigma$ such that $|\Sigma|_{g_{\mathbb S^3}}\to 0$ without affecting other terms in \eqref{est'}. However, we do not employ this observation in our numerical verification.

We now consider the family of ellipsoids $\{\frac{x^2}{a^2}+\frac{y^2}{b^2}+\frac{z^2}{c^2}=1\}$, where the genus $g=0$.
We compare the differences between the two estimates
$E_2=\lfloor 1-\frac{1}{4\pi}|\Sigma|_{g_{\mathbb S^3}}+\frac{1}{8\pi}\int_\Sigma |\stackrel \circ {\overline A}|^2 dS_{\overline g}\rfloor $ and
$E_f=\lfloor\frac{1}{4\pi^2}\int_\Sigma f\left(|\stackrel \circ { A}|\right) dS_{ g_{\mathbb S^3}} \rfloor$, where $\lfloor \cdot\rfloor$ is the floor function.

Using Mathematica\footnote{The computation file is available upon request. }, we have computed the differences $E_2-E_f$ for all $a, b, c$ ranging between ${1, \cdots, 10}$. Our experiment demonstrates that for all $a, b, c$ in this range, $E_2-E_f\ge 0$, and the differences can be as large as $16$ (occurs when $(a, b, c)=(1, 10, 10)$). Consequently, our estimate \eqref{est} is better than the classical estimate \eqref{est'}.

\end{document}